\documentclass[]{article}

\usepackage{amsfonts,amssymb,amsmath,amsthm}
\usepackage{graphicx}
\usepackage{geometry}
\usepackage{xcolor}
\usepackage[colorlinks=true,linkcolor=blue,citecolor=green!50!black]{hyperref}
\usepackage{enumerate}

\newtheoremstyle{custom}{3pt}{3pt}{}{}{\bfseries}{:}{.5em}{}
\theoremstyle{custom}
\newtheorem{example}    {Example}

\newtheorem{theorem}    [example]{Theorem}    
\newtheorem{lemma}      [example]{Lemma}
\newtheorem{corollary}  [example]{Corollary}

\newtheorem{remark} 		[example]{Remark}

\newcommand{\R}{\mathbb{R}}

\newcommand{\N}{\mathbb{N}}

\newcommand{\A}{\mathcal{A}}
\newcommand{\G}{\mathcal{G}}
\newcommand{\D}{\mathcal{D}}
\renewcommand{\P}{\mathcal{P}}
\newcommand{\T}{\mathcal{T}}

\def \tx {\widetilde{X}}

\def \tPt {\widetilde{\P}^t}
\def \tphi {\widetilde{\phi}}
\def \tG {\widetilde{\G}}

\newcommand{\inter}{\mathrm{int}\, }

{
\setbox1\hbox{$
\global\fontdimen16\textfont2=3.5pt
\global\fontdimen17\textfont2=3.5pt
\global\fontdimen16\scriptfont2=3pt
\global\fontdimen17\scriptfont2=3pt
\global\fontdimen16\scriptscriptfont2=2.7pt
\global\fontdimen17\scriptscriptfont2=2.7pt$}
}

\title{Discrete infinitesimal generator\\ of the Frobenius--Perron operator semigroup\\ associated with ``outflow systems''}
\author{P\'eter~Koltai}

\begin{document}

\maketitle
\begin{abstract}
In this technical report the $C_0$ semigroup of Frobenius--Perron operators on $L^1(X)$ is considered, where the underlying dynamical system is such that trajectories may leave the state space $X$ and terminate. We introduce a discrete infinitesimal generator and show, that the operator semigroup generated by this discrete generator converges in $L^1(X)$ pointwise to the Frobenius--Perron operator of the system.
\end{abstract}

\section{How to read this report}

For a deeper insight into the motivation of the work done here we refer to~\cite{FroJK10}, \cite{KPPHD} Chapter~5, and~\cite{Kol10}. Also, the proofs presented below are essentially based on proofs in~\cite{KPPHD}, and only the differences between them are discussed here in detail and with mathematical rigor. Nevertheless, Section~\ref{sec:intro}, the introduction, can be read with some basic knowledge on semigroups of linear operators (see, e.g.~\cite{Pazy83}), and on the theory of transfer operators (see, e.g.~\cite{LaMa94}). The main result is Corollary~\ref{cor:main}.

\section{Introduction}	\label{sec:intro}

\paragraph{The outflow system.}

Let $X\subset\R^d$ be a compact set with nonempty interior, such that the outward pointing normal vector $\mathbf{n}$ exists almost everywhere on $\partial X$ w.r.t.\ the $d-1$ dimensional Lebesgue measure $m_{d-1}$. Further, let ${v\in C^2(\R^d,\R^d)}$ be a vector field defining the flow $\phi^t$ by
\[
\frac{d}{dt}\phi^t\cdot = v\left(\phi^t\cdot\right).
\]
We would like to track redistribution of mass caused by the dynamics \textit{only in} $X$. What leaves $X$ is ignored and considered as lost. Hence, the corresponding dynamical system is defined as
\[
\tphi^t x = \left\{ \begin{array}{ll}
											\phi^t x, & \text{if }\phi^{\text{sign}(t) s} x\in\inter X\text{ for all }s\in[0,|t|], \\
											\text{lost}, & \text{otherwise},
										\end{array}\right.
\]
where $\tphi^t\text{lost}=\text{lost}$ for all $t\ge0$, and its state space is given by $\tx:=X\cup\{\text{lost}\}$. We call the restriction of this system to $X$, i.e.\ the pair $(X,\tphi^t\vert_X)$, the \textit{outflow system}. For simplicity, we denote $\tphi^t\vert_X$ again by $\tphi^t$. The associated transfer operator ${\tPt:L^1( X)\to L^1( X)}$, $t\ge0$, is given by
\begin{equation}
\tPt u(x) = \left\{ \begin{array}{ll}
											u\left(\phi^{-t} x\right)\left|\det\left(D_x\phi^{-t}x\right)\right|, & \text{if }\phi^{-s} x\in\inter X\text{ for all }s\in[0,t], \\
											0, & \text{otherwise},
										\end{array}\right.
\label{eq:FPOout}
\end{equation}
where $D_x$ denotes the derivative w.r.t.\ the spatial ($x$) coordinate. It is seen easily that $\{\tPt\}_{t\ge0}$ is an operator semigroup. We denote its infinitesimal generator by $\tG$, and the domain of $\tG$ by $\D(\tG)$. Note, that ${C_v^1( X)\subset\D(\tG)}$, with
\[
C_v^1( X) = \left\{u\in C^1( X)\big\vert u=0\text{ on }\{v\cdot\mathbf{n}<0\}\subset\partial X\right\},
\]
where $x\cdot y$ denotes for $x,y\in\R^d$ the usual inner product. In general, $C^1(X)\nsubseteq \D(\tG)$.

\paragraph{The discretization.}

Let ${\{X_n\}_{n\in\N}}$ be a family of sets resulting from uniform\footnote{A partition (covering) of a set is called \textit{uniform} if all partition (covering) elements are mutually congruent.} box coverings (we call a hyperrectangle in $\R^d$ a box) with fineness tending to zero, such that $ X\subseteq X_n$ for all $n$. The set $X_n$ is the union of disjoint boxes $X_{n,i}$, where $X\cap X_{n,i}\neq\emptyset$ for all $n,i$. Let $V_n$ be the space of piecewise constant functions on the corresponding covering, and ${\pi_n:L^1(\R^d)\to V_n}$ the $L^2$-orthogonal projection to $V_n$, i.e.
\[
\pi_n u = \sum_{i} \left(\frac{1}{m(X_{n,i})}\int_{X_{n,i}}\!\!u\right)\chi_{X_{n,i}}.
\]
Note, that $\pi_nu$ vanishes outside of $X_n$.

We wish to discretize the generator $\tG$ on $X_n$. The difficulty that $\tG$ operates on functions supported only on $X\subset X_n$ is solved as follows. We define the discrete generator ${\tG_n:L^1(\R^d)\to V_n}$ by means of the original semigroup $\P^t$:
\begin{equation}
\tG_n u = \lim_{t\to0}\frac{\pi_n\P^t\pi_n u-\pi_n u}{t}.
\label{eq:IGout}
\end{equation}
That the right hand side of~\eqref{eq:IGout} exists for all $u\in L^1(\R^d)$ is a simple corollary of lemmas~5.9 and 5.11 in~\cite{KPPHD}.

In this report we show that the semigroup generated by $\tG_n$ on $L^1(X)$ converges pointwise to $\tPt$ as $n\to\infty$ for every fixed $t\ge0$.

\section{Convergence proof}

We claim the convergence based on the following theorem. The lemmas below prove that its assumptions are valid for our setting.
\begin{theorem}[Theorem~3.4.5 \cite{Pazy83}]	\label{T:ig_approx}
Let $\A_n\in G(M,\omega)$ and assume:
\begin{enumerate}[(a)]
	\item As $n\to\infty$, $\A_n u\to \A u$ for every $u\in D$, where $D$ is a dense subset of $L^1( X)$.
	\item There exists a $\lambda_0$ with Re$\, \lambda_0>\omega$ for which $(\lambda_0I-\A)D$ is dense in $L^1( X)$.
\end{enumerate}
Then, the closure $\bar{\A}$ of $\A$ is in $G(M,\omega)$. If $\T_n(t)$ and $\T(t)$ are the $C_0$ semigroups generated by $\A_n$ and $\bar{\A}$ respectively, then
\[
\lim_{n\to\infty}\T_n(t) u = \T(t)u\qquad\text{ for all }t\ge 0, u\in L^1( X),
\]
and this limit is uniform in $t$ for $t$ in bounded intervals.
\end{theorem}

The restriction of $\tG_n$ to $X$ is understood as follows. Every $u\in L^1(X)$ can be extended by the value zero outside of $X$ to obtain a $\tilde u\in L^1(\R^d)$. Then $\tG_nu$ is defined as the pointwise restriction of $\tG_n\tilde u$ to $X$.
\begin{center}
\fbox{\parbox[l]{15cm}{
\begin{corollary}[Convergence of the generated outflow semigroup]	\label{cor:main}
The semigroup generated by the restriction of $\tG_n$ to $ X$ converges pointwise in $L^1$ against $\tPt$ as $n\to\infty$ for every fixed $t\ge0$.
\end{corollary}
}}
\end{center}
\begin{proof}
It can be shown as in Theorem~5.21~\cite{KPPHD} that $\tG_n\in G(1,0)$ for all $n$. Set $D=C_{0,0}^1( X)$ (see the definition below). Assumption (a) of Theorem~\ref{T:ig_approx} holds by Lemma~\ref{L:IGconv}. Assumption (b) holds by Lemma~\ref{L:inclusion} (b), since it gives $C_{0,0}^1( X)\subseteq (\lambda-\tG)C_{0,0}^1( X)$, and $C_{0,0}^1( X)$ is dense in $L^1( X)$.
\end{proof}

Even though we just restricted the operator $\tG_n$ to $X$, it will prove useful to extend the support of the functions under consideration. Define
\[
C_{X}^1(\R^d):=\left\{f\in C^1(\R^d)\Big\vert \mathrm{supp}\, f\subset X \right\},
\]
and its restriction to $X$
\[
C_{0,0}^1( X) = \left\{f\big\vert_{X}\ \Big\vert f\in C_{ X}^1(\R^d)\right\}.
\]
Let $u\in C_X^1(\R^d)$ and denote $\tilde u=u\vert_X$. Observe that $\tG \tilde u = \left(\G u\right)\!\big\vert_X$, where $\G$ is the generator of the semigroup of Frobenius--Perron operators associated with $\phi^t$. 

\begin{lemma}[cf.~\cite{KPPHD} Lemma~5.17]	\label{L:IGconv}

\quad

We have $\big\|\tG_n u-\tG u\big\|_{L^1( X)}\to0$ as $n\to\infty$ for $u\in C_{0,0}^1(X)$.
\end{lemma}
\begin{proof}
We show now $\big\|\tG_n u-\G u\big\|_{\infty}\to0$ for $u\in C_X^1(\R^d)$, which implies the claim.

The proof of Lemma~5.17 in~\cite{KPPHD} estimates the difference between $\G u$ and $\tG_n u$ locally for each box in the covering. There, the functions $u$ and $v$ are bounded on the compact state space, just as their derivatives, hence these local estimates apply globally, and the convergence is shown.

Here, the local estimates can be derived analogously, by noting the following. If the box under consideration is on the boundary of the covering, from ${\R^d\setminus X_n}$ there is no inflow into the box, since $u$ vanishes outside $X_n$. For the same reason, outside the covering both $\tG_nu$ and $\G u$ are zero.

Since ${X\cap X_{n,i}\neq\emptyset}$ for all $n,i$, the state space $X$ is compact and the box diameters tend to zero as $n\to\infty$, there is a compact set containing $\bigcup_{n\in\N}X_n$. Hence, $u, D_x u, v, D_x v$ and $D_x^2v$ are bounded functions on $\bigcup_{n\in\N}X_n$. By this, the local estimates apply globally. This completes the proof.
\end{proof}

\begin{lemma}[cf.~\cite{KPPHD} Lemma~5.20]	\label{L:inclusion}
\quad
\begin{enumerate}[(a)]
	\item For a sufficiently large $\lambda>0$ one has $(\lambda-\tG)^{-1}u\in C_v^1( X)$ for $u\in C_v^1( X)$.
	\item For a sufficiently large $\lambda>0$ one has $(\lambda-\tG)^{-1}u\in C_{0,0}^1( X)$ for $u\in C_{0,0}^1( X)$.
\end{enumerate}
\end{lemma}
\begin{remark}
Note, that from $u\in C_{0,0}^1( X)$ follows $u=0$ and $D_xu=0$ on $\partial X$.
\end{remark}
\begin{proof}
It holds ${(\lambda-\tG)^{-1}u = \int_0^{\infty}e^{-\lambda t}\tPt u\, dt}$; see Remark~1.5.4 in~\cite{Pazy83}. To show the claims analogously to the proof of Lemma~5.20 in~\cite{KPPHD}, we have to show differentiability of the function ${x\mapsto\tPt u(x)}$ w.r.t.\ $x$ for all $t$, and the integrability of ${e^{-\lambda t}\tPt u(x)}$ w.r.t.\ $t$ for all $x\in X$. This will imply ${(\lambda-\tG)^{-1}u\in C^1( X)}$. Additionally, the boundary conditions have to be checked.\\
(a). The function ${x\mapsto\tPt u(x)}$ is continuously differentiable for almost every $t\ge0$ with derivatives bounded exponentially in $t$, since $u$ is zero on the ``inflow'' part of the boundary; cf.~\eqref{eq:FPOout} (this shows the integrability for a $\lambda$ sufficiently large). There can be only one $t$ for which $\tPt u(x)$ is not differentiable: the smallest $t$ such that $\tphi^{-t}x\in\partial X$. However, it is still continuous, hence the subdifferentials are bounded at this point, so the above integral exists and is a function differentiable in $x$. Since $\tPt u(x)=0$ for all $t\ge0$ and $x\in\{v\cdot\mathbf{n}<0\}$, we have $(\lambda-\tG)^{-1}u(x)=0$ for $x\in\{v\cdot\mathbf{n}<0\}$ as well.

\noindent (b). The function ${x\mapsto\tPt u(x)}$ is continuously differentiable for every $t\ge0$ with derivatives bounded exponentially in $t$, since $u$ and $D_xu$ is zero on the whole boundary (this is the only point, where a discontinuity like in (a) could arise, but the choice of ${u\in C_{0,0}^1( X)}$ makes this impossible). Applying the chain rule to~\eqref{eq:FPOout}, we see that $u\big\vert_{\partial X}=0$ and $D_xu\big\vert_{\partial X}=0$ implies ${\tPt u\big\vert_{\partial X}=0}$ and ${D_x\tPt u\big\vert_{\partial X}=0}$. This completes the proof.
\end{proof}

\section{Conclusion}

We proved the $L^1$-pointwise convergence of the semigroup generated by the approximate generator, defined in~\eqref{eq:IGout}, to the semigroup of Frobenius--Perron operators, given in~\eqref{eq:FPOout}, associated with the outflow system.

The proof was based on techniques used in~\cite{KPPHD}, and included a crucial assumption, namely that all boxes in a covering are congruent. Without this, we don't have the convergence ${\tG_nu\to\tG u}$, as the example in Remark~5.18~\cite{KPPHD} shows. However, numerical experiments suggest that the claim of Corollary~\ref{cor:main} remains true, as long as the diameter of the largest box in the covering tends to zero, and the other assumptions are valid. To prove this, it seems to be necessary to show the convergence of the resolvents, i.e.\ $\big(\lambda-\tG_n\big)^{-1}\to\big(\lambda-\tG\big)^{-1}$, instead of ``just'' the convergence of the generators.

Subject of future work is also the extension of the result obtained here to non-deterministic systems given by the stochastic differential equation
\[
	dX = v(X)dt + dW,
\]
where $W$ denotes the Brownian motion.

\bibliographystyle{alpha}

%
%
%
%
%
%

\bibliography{../../Literatur}

\end{document}